\documentclass[11pt]{35}
\usepackage{amsthm, amsmath, amssymb, amsbsy, amsfonts, latexsym, euscript}
\usepackage{calrsfs} 
\usepackage{graphicx}

\DeclareMathAlphabet{\varmathbb}{U}{pxsyb}{m}{n}

\def\leq{\leqslant}

\def\geq{\geqslant}
\def\phi{\varphi}

\def\bar{\overline}

\def\kappa{\varkappa}

\newcommand{\D}{\mathrm{d}\kern0.2pt}%

\newcommand{\E}{\mathrm{e}\kern0.2pt} 
\newcommand{\ii}{\kern0.05em\mathrm{i}\kern0.05em}

\newcommand{\RR}{\mathbb{R}}%

\newtheorem{theorem}{\bf \indent Theorem}[section]
\newtheorem{proposition}{\bf \indent Proposition}[section]
\newtheorem{lemma}{\bf \indent Lemma}[section]
\newtheorem{corollary}{\bf \indent Corollary}[section]

\theoremstyle{remark}

\newtheorem{conjecture}{\bf \indent Conjecture}[section]

\numberwithin{equation}{section}

\begin{document}

\noindent {\Large \bf INVERSE MEAN VALUE PROPERTIES (A SURVEY)}

\vskip5mm

{\bf Nikolay Kuznetsov}

\vskip-2pt {\small Laboratory for Mathematical Modelling of Wave Phenomena}
\vskip-4pt {\small Institute for Problems in Mechanical Engineering, Russian Academy
of Sciences} \vskip-4pt {\small V.O., Bol'shoy pr. 61, St Petersburg, 199178,
Russian Federation} \vskip-4pt {\small nikolay.g.kuznetsov@gmail.com}

\vskip4mm

\parbox{134mm} {\noindent Several mean value identities for harmonic and panharmonic
functions are reviewed along with the corresponding inverse properties. The latter
characterize balls, annuli and strips analytically via these functions.}

{\centering \section{Introduction and notation} }

\noindent Inverse mean value properties of harmonic functions ($C^2$-solutions of
the Laplace equation; see \cite{ABR}, p.~25, about the origin of the term
`harmonic') are well known for balls and spheres; see the survey article \cite{NV},
Sections~7 and 8, respectively. In particular, Kuran \cite{K} obtained a simple
proof of the following general assertion.


\begin{theorem}
Let $D$ be a domain (= connected open set) of finite (Lebesgue) measure in the
Euclidean space $\RR^m$ where $m \geq 2$. Suppose that there exists a point $P_0$
in $D$ such that, for every function $h$ harmonic in $D$ and integrable over $D$,
the volume mean of $h$ over $D$ equals $h (P_0)$. Then $D$ is an open ball (disk
when $m=2$) centred at $P_0$.
\end{theorem}


\noindent His original proof is available online at
https://matthewhr.wordpress.com/2014/11/15/ kurans-theorem/

At the same time, many interesting results in this area were not reviewed so far.
They concern other characterizations of balls via harmonic functions as well as
characterization of various other domains: strips, annuli etc. Moreover, it was,
until recently, unknown whether an assertion similar to Theorem~1.1 is true if
solutions of another partial differential equation are used instead of harmonic
functions; of course, the equality for the arithmetic mean over balls must be
adjusted to these solutions. Only in 2021, a characterization of $m$-dimensional
balls by solutions to the modified Helmholtz equation
\begin{equation}
\nabla^2 u - \mu^2 u = 0 , \quad \mu \in \RR \setminus \{0\} ,
\label{Hh}
\end{equation}
was obtained; see \cite{Ku, Ku1}. (Here and below, $\nabla = (\partial_1, \dots ,
\partial_m)$, $\partial_i = \partial / \partial x_i$, denotes the gradient
operator.) In what follows, instead of the cumbersome `solution of the modified
Helmholtz equation' the term `panharmonic function' is used; this convenient
abbreviation was introduced in \cite{D}.

The are numerous and diverse results about inverse mean value properties that are
not covered in \cite{NV}, and our aim is to review them; the plan is as follows.
Mean value equalities used in the paper are described in Section~2 with proper
references. In Section~3, these equalities are applied for characterizing balls via
harmonic and panharmonic functions, whereas characterizations of annuli and strips
by quadrature formulae involving mean values are considered in Section~4. It should
be said that caloric functions (solutions of the heat equation) are beyond our scope
because their mean value properties have rather specific character different from
that considered in Section~2; see, for example, \cite{F,NAW}.

Let us introduce some notation used below. For a point $x = (x_1, \dots, x_m) \in
\RR^m$, $m \geq 2$, we denote by $B_r (x) = \{ y \in \RR^m : |y-x| < r \}$ the open
ball of radius $r$ centred at $x$ (just $B_r$, if centred at the origin). The ball
is called admissible with respect to a domain $D \subset \RR^m$ provided
$\overline{B_r (x)} \subset D$, and $\partial B_r (x)$ is called admissible sphere
in this case. If $D$ has finite Lebesgue measure and a function $f$ is integrable
over $D$, then\\[-3mm]
\[ M^\bullet (f, D) = \frac{1}{|D|} \int_{D} f (x) \, \D x
\]
is its volume mean value over $D$. Here and below $|D|$ is the domain's volume (area
if $D \subset \RR^2$), and the volume of $B_r$ is $|B_r| = \omega_m r^m$, where
$\omega_m = 2 \, \pi^{m/2} / [m \Gamma (m/2)]$ is the volume of the unit ball; as
usual $\Gamma$ denotes the Gamma function. For $u \in C^0 (\overline D)$, the mean
value\\[-5mm]
\[ M^\circ (f, \partial D) = \frac{1}{|\partial D|} \int_{\partial D}
f (y) \, \D S_y 
\]
over a sufficiently smooth $\partial D$ is another useful notion; here $|\partial
D|$ is the surface area of domain's boundary ($|\partial B_r| = m \, \omega_m
r^{m-1}$), and $\D S$ is the surface area measure. The exterior unit normal on
$\partial D$ is denoted by $n$.

\vspace{-4mm}

{\centering \section{Mean value equalities} }

\vspace{-2mm}

\subsection{Spheres and balls}

\vspace{-1mm}

\noindent Studies of mean value properties of harmonic functions date back to the
Gauss theorem of the arithmetic mean over a sphere; see \cite{G}, Article~20.
Nowadays, its standard formulation is as follows.

\vspace{-2mm}

\begin{theorem}
Let $D$ be a domain in $\RR^m$, $m \geq 2$. Then $u \in C^2 (D)$ is harmonic in~$D$
if and only if for every $x \in D$\\[-5mm]
\begin{equation}
M^\circ (u, \partial B_r (x)) = u (x) \label{Ms}
\end{equation}
for each admissible sphere $\partial B_r (x)$.
\end{theorem}

\vspace{-2mm}

The assertion that \eqref{Ms} implies that $u$ is harmonic was proved by Koebe
\cite{Ko} in 1906. Integrating \eqref{Ms} with respect to $r$ over $(0, R)$, one
obtains the following.

\vspace{-1mm}

\begin{corollary}
Let $u \in C^2 (D)$ be harmonic in a domain $D \subset \RR^m$, $m \geq 2$.
Then\\[-3mm]
\begin{equation}
M^\bullet (u, B_R (x)) = u (x)
\label{Mb}
\end{equation}
for every $x \in D$ and each admissible ball $B_R (x)$. The converse is also true.
Moreover,\\[-3mm]
\begin{equation}
M^\bullet (u, B_R (x)) = M^\circ (u, \partial B_R (x)) \label{Mbs}
\end{equation}
for every such ball.
\end{corollary}

Let us turn to analogous equalities for panharmonic functions; they involve the
following coefficients
\begin{equation} 
a^\circ (\mu r) = \Gamma \left( \frac{m}{2} \right) \frac{I_{(m-2)/2} (\mu r)} {(\mu
r / 2)^{(m-2)/2}} \quad \mbox{and} \quad a^\bullet (\mu r) = \Gamma \left(
\frac{m}{2} + 1 \right) \frac{I_{m/2} (\mu r)}{(\mu r / 2)^{m/2}} \, , \label{cibu}
\end{equation}
where $I_\nu$ denotes the modified Bessel function of order $\nu$.


\begin{theorem}
Let $D$ be a domain in $\RR^m$, $m \geq 2$. Then $u \in C^2 (D)$ is panharmonic in
$D$ if and only if for every $x \in D$
\begin{equation}
M^\circ (u, \partial B_r (x)) = a^\circ (\mu r) \, u (x) \label{Mps}
\end{equation}
for each admissible sphere $\partial B_r (x)$.
\end{theorem}


Formula \eqref{Mps} has particularly simple form for $m=3$ because $a^\circ (\mu r)
= \sinh \mu r / (\mu r)$, which was proved by C.~Neumann \cite{NC} as early as 1896.
Duffin independently rediscovered his proof (see \cite{D}, pp.~111--112), but for
the two-dimensional case with $a^\circ (\mu r) = I_0 (\mu r)$. Finally, the author
\cite{Ku2} derived \eqref{Mps} for any $m \geq 2$ and announced \cite{Ku1} the
convrse assertion; its proof is given below.


\begin{proof}[A sketch of the proof of Theorem 2.2]
First, we outline the derivation of \eqref{Mps} for a panharmonic $u$. It based on
the Euler--Poisson--Darboux equation
\begin{equation}
M^\circ_{rr} + (m-1) \, r^{-1} M^\circ_r = \nabla^2_x M^\circ  \label{EPD}
\end{equation}
satisfied by $M^\circ (u, \partial B_r (x))$ for $r > 0$ (see \cite{J}, Chapter IV).

By virtue of the following relations
\begin{equation}
z I_{\nu+1} (z) + 2 \nu I_\nu (z) - z I_{\nu-1} (z) = 0 \, , \ \ \ [z^{-\nu} I_\nu
(z)]' = z^{-\nu} I_{\nu+1} (z)  \label{diff}
\end{equation}
(see \cite{Wa}, p. 79), it is easy to show that a unique solution of the Cauchy
problem
\begin{equation}
a_{rr} + (m-1) \, r^{-1} a_r - \mu^2 a = 0 , \ \ a (0) = 1 , \ \ a_r (0) = 0 \, ,
\label{g5}
\end{equation}
is $a (r) = a^\circ (\mu r)$. Combining \eqref{Hh} and \eqref{EPD}, we see that
$a^\circ (\mu r) \, u (x) - M^\circ (u, \partial B_r (x))$ satisfies equation in
\eqref{g5} with zero initial conditions. Hence it vanishes identically, thus
yielding~\eqref{Mps}.
\end{proof}


Prior to proving the converse assertion, let us consider some consequences of
\eqref{Mps}. Since $a^\circ (0) = 1$, the first initial condition \eqref{g5} yields
that \eqref{Mps} turns into \eqref{Ms} as $\mu \to 0$. To determine how the mean
value of a panharmonic function $u$ depends on radii of admissible spheres centred
at an arbitrary point $x \in D$, we notice that the behaviour of $M^\circ (u,
\partial B_r (x))$ is the same as that of $a^\circ (\mu r)$. The latter is
continuous and increases monotonically by the second formula \eqref{diff}; hence
\begin{equation*}
|M^\circ (u, \partial B_r (x))| > |u (x)| \ \ \mbox{for} \ r > 0 .
\label{sub}
\end{equation*}
Moreover, Poisson's integral for $I_\nu$ (see \cite{NU}, p.~223) implies that:
\begin{equation}
a^\circ (\mu r) = \int_0^1 (1 - s^2)^{(m - 3)/2} \cosh (\mu r s) \, \D s \, .
\label{PI}
\end{equation}
Hence, $M^\circ (u, \partial B_r (x))$ is a convex function of $r$.

Integrating equality \eqref{Mps} with respect to $r$ over $(0, R)$ and using
formula
\begin{equation*}
\int_0^x \!\! x^{1 + \nu} I_\nu (x) \, \D x = x^{1 + \nu} I_{\nu + 1} (x) \, , \ \ \Re
\, \nu > -1 . \label{PBM}
\end{equation*}
(see 1.11.1.5, \cite{PBM}) with $\nu = (m-2)/2$, one obtains the following.


\begin{corollary}
Let $u \in C^2 (D)$ be panharmonic in a domain $D \subset \RR^m$, $m \geq 2$. Then
\begin{eqnarray}
&& \ \ \ \ \ \ \ \ \ \ \ \ M^\bullet (u, B_R (x)) = a^\bullet (\mu R) \, u (x) \, ,
\label{Mpb} \\ && a^\circ (\mu R) \, M^\bullet (u, B_R (x)) = a^\bullet (\mu R) \,
M^\circ (u, \partial B_R (x)) \label{Mpbs}
\end{eqnarray}
for every $x \in D$ and each admissible ball $B_R (x)$. The converse is also true.
\end{corollary}


As $\mu \to 0$, the volume mean formula \eqref{Mpb} for a panharmonic function turns
into \eqref{Mb} for a harmonic one. Moreover, the dependence on $r$ for the volume
mean is the same as for the spherical one. Finally, equality \eqref{Mpbs} (a
consequence of \eqref{Mpb} and \eqref{Mps} with $r = R$) implies that $M^\bullet (u,
B_R (x)) < M^\circ (u, \partial B_R (x))$; indeed,
\[ \frac{M^\bullet (u, B_R (x))}{M^\circ (u, \partial B_R (x))} =
\frac{a^\bullet (\mu R)}{a^\circ (\mu R)} < 1 \, ,
\]
which immediately follows from the definition of $a^\bullet$ and $a^\circ$ and the
first formula \eqref{diff}.


\begin{proof}[A sketch of the proof of Theorem 2.2 (continued)]
To show that $u$ is panharmonic in $D$ when \eqref{Mps} is valid, we notice that a
consequence of \eqref{Mps}, namely, \eqref{Mpb} (with $R$ changed to $r$) is
equivalent to
\begin{equation}
\left( \frac{2 \pi r}{\mu} \right)^{m/2} I_{m/2} (\mu r) \, u (x) = \int_{|y| < r}
\!\! u (x+y) \, \D y . \label{Rh}
\end{equation}
Applying the Laplacian to the integral on the right-hand side, we obtain
\[ \int_{|y| < r} \!\! \nabla^2_x \, u (x+y) \, \D y = \int_{|y| = r} \!\! \nabla_x \,
u (x+y) \cdot \frac{y}{r} \, \D S_y \, .
\]
Here the equality is a consequence of Green's first formula. By changing variables
this can be written as follows:
\[ r^{m-1} \frac{\partial}{\partial r} \int_{|\theta|=1} \!\! u (x + r \theta) \, 
\D S^{m-1}_\theta = |S^{m-1}| r^{m-1} \frac{\partial}{\partial r} M^\circ (u,
\partial B_r (x)) \, .
\]
In view of \eqref{Mps} and the second formula \eqref{diff}, we have that
\[ \frac{\partial}{\partial r} M^\circ (u, \partial B_r (x)) = \frac{\mu I_{m/2} 
(\mu r)}{(\mu r / 2)^{(m-2)/2}} \, u (x) \, .
\]
Combining the above considerations and \eqref{Rh}, we conclude that
\[ \int_{|y| < r} \!\! [ \nabla^2_x \, u - \mu^2 u ] \, (x+y) \, \D y = 0
\]
for every $x \in D$ and all $r$ such that $B_r (x)$ is admissible. Hence, there
exists $y (r, x) \in B_r (x)$ such that $[ \nabla^2 \, u - \mu^2 u ] \, (y (r, x)) =
0$. Since $y (r, x) \to x$ as $r \to 0$, it follows by continuity that $u$ is
panharmonic in $D$.
\end{proof}


\subsection{Annuli}

Let $m \geq 3$ and $r_2 > r_1 \geq 0$; following Armitage and Goldstein \cite{AG},
we introduce
\begin{equation}
r_* = \left( \frac{2}{m} \frac{r_2^m - r_1^m}{r_2^2 - r_1^2} \right)^{1/(m-2)} .
\label{r*}
\end{equation}
The convexity of $t \mapsto t^{m-2}$ implies that $r_* \in (r_1, r_2)$, and for this
reason the domain
\[ A (r_1, r_2) = \{ x \in \RR^m : r_1 < |x| < r_2 \} \, ,
\]
is called an $r_*$-annulus (below, we write $A$ for brevity). For $m = 2$, such an
annulus has
\begin{equation}
r_* = \exp \left( \frac{r_2^2 \log r_2 - r_1^2 \log r_1}{r_2^2 - r_1^2} -
\frac{1}{2} \right) ;
\label{r*2}
\end{equation}
see \cite{GS}, p.~437. Now, we are in a position to formulate the following.


\begin{theorem}
Let $A \subset \RR^m$, $m \geq 2$, be an $r_*$-annulus. If $u \in C^2 (A)$ is
harmonic and integrable over $A$, then
\begin{equation}
M^\bullet (u, A) = M^\circ (u, \partial B_{r_*}) \, . \label{Ma}
\end{equation}
\end{theorem}


\begin{proof}[Proof for $m \geq 3$, {\rm \cite{AG}}.]
For arbitrary $r_1', r_2'$ such that $r_1 <r_1' < r_2' < r_2$, we have
\[ \Bigg[ \int_{\partial B_{r_2'}} - \int_{\partial B_{r_1'}} \Bigg] \partial u /
\partial n \, \D S = 0 \, ,
\]
because $u$ is harmonic. This implies that
\[ r^{m-1} \D M^\circ (u, \partial B_{r}) / \D r = \mathrm{constant} \quad \mbox{for}
\ r \in (r_1, r_2) \, ,
\]
and so $M^\circ (u, \partial B_{r}) = c_1 r^{2-m} + c_2$, where $c_1$ and $c_2$ are
constants. Hence
\[ M^\bullet (u, A) = \frac{m}{r_2^m - r_1^m} \int_{r_1}^{r_2} M^\circ (u, \partial B_r)
\, r^{m-1} \, \D r = c_1 \frac{m}{2} \frac{r_2^2 - r_1^2}{r_2^m - r_1^m} + c_2 \, ,
\]
from which \eqref{Ma} follows in view of \eqref{r*}.
\end{proof}


\subsection{Strips}

In 1992, Armitage and Goldstein \cite{AG1} (see also \cite{AG3}) obtained the
following mean value property of harmonic functions on an infinite open strip $S (a,
b) = (a, b) \times \RR^m \subset \RR^{m+1}$; here $(a, b)$ is a bounded open
interval on, say, $t$-axis.


\begin{theorem}
Let $m \geq 2;$ if a real-valued function $u$ is harmonic and integrable on $S (a,
b)$, then
\begin{equation}
\int_{S (a, b)} u (t, x) \, \D t \D x = (b - a) \int_{\RR^m} u ( (a+b) / 2, x) \, \D
x \, . \label{Mst}
\end{equation}
\end{theorem}


\begin{proof}[A sketch of the proof]
Armitage and Gardiner \cite{AG2} (see also papers cited therein) investigated
hyperplane means by virtue of Fubini's theorem and the fact that $|u|$ is
subharmonic. According to these results, there exists
\[ \int_{\RR^m} u (t, x) \, \D x \quad \mbox{for} \ t \in (a, b) \, ,
\]
and this function is a first degree polynomial in $t$, which implies \eqref{Mst}.
\end{proof}


Goldstein, Haussmann and Rogge \cite{GHR} extended this result to an
$m+l$-dimensional, $m,l \geq 2$, bi-infinite cylinder $C_r = B_r \times \RR^l_y$,
where $B_r$ is the $m$-dimensional open ball centred at the origin of $\RR^m_x$;
namely, they proved the following.


\begin{theorem}
Let $C_r \subset \RR^{m+l}_{x,y}$ be a bi-infinite open cylinder. If a real-valued
$u$ is harmonic and integrable on $C_r$, then
\begin{equation*}
\int_{C_r} u (x, y) \, \D x \D y = |B_r| \int_{\RR^l} u (0, y) \, \D y \, .
\label{MC}
\end{equation*}
\end{theorem}


It is interesting whether the results of Sections~2.2 and 2.3 have analogues for
panharmonic functions.


{\centering \section{Characterization of balls via mean value properties} }

\subsection{Harmonic functions}

\noindent Along with Kuran's theorem, there are other characterizations of balls
based on mean value properties. The best known is the following assertion of
Aharonov, Schiffer and Zalcman \cite{ASZ}.

\begin{theorem} 
Let $D \subset \RR^m$, $m \geq 3$, be a bounded open set such that $D = \mathrm{int}
D$ and $\partial D$ has zero volume. For some fixed $x_0 \in D$ the identity
\begin{equation*}
|x_0 - y|^{2-m} = \frac{1}{|D|} \int_{D} |x - y|^{2-m} \, \D x
\end{equation*}
is valid for every $y \in \RR^m \setminus D$ if and only if $D = B_r (x_0)$, where
$r$ is such that $|D| = |B_r (x_0)|$.
\end{theorem}

The authors assumed $m = 3$ in this theorem, but it is straightforward to extend the
result to higher dimensions. Netuka and Vesel\'y \cite{NV}, p.~377, just mentioned
the note \cite{ASZ} as an application of Kuran's theorem to answering the question:
Must a bounded homogeneous solid in $\RR^3$, ``which gravitationally attracts each
point outside it as if all its mass were concentrated at a single point, [\dots] be
a ball?'' (see \cite{ASZ}, p.~331). Further discussion of this and related results can
be found in Cupini and Lanconelli's article \cite{CL}; see also references cited
therein.

Furthermore, Netuka and Vesel\'y formulated a two-dimensional version (see
\cite{NV}, Section~8) of the following general assertion.


\begin{theorem} 
Let $D \subset \RR^m$, $m \geq 2$, be a bounded $C^2$-domain. If
\begin{equation}
\frac{1}{|D|} \int_{D} h (x) \, \D x = \frac{1}{|\partial D|} \int_{\partial D} h
(x) \, \D S_x \label{3.1}
\end{equation}
for every harmonic in $D$ function $h \in C^1 (\overline D)$, then $D$ is a ball.
\end{theorem}


Equality \eqref{3.1} closely resembles the relation between the mean values over a
ball and its boundary; cf. \eqref{Mbs}. The theorem cited in \cite{NV} was published
by Kosmodem'yanskii~\cite{Kos} five years earlier than Theorem~3.2 was proved by
Bennett \cite{Be} in 1986. However, Kos\-mo\-dem'yanskii's result has two drawbacks:
it is essentially two-dimensional and the super\-fluous assumption that $D$ is
convex is imposed. The latter condition arises in an auxiliary assertion similar to
the following theorem of Serrin \cite{S} and Weinberger~\cite{W}, on which Bennett's
proof is based.


\begin{theorem}[Serrin, Weinberger]
Let $D \subset \RR^m$ be a bounded $C^2$-domain and let
\begin{equation}
\nabla^2 u = - 1 \ in \ D \quad and \quad u = 0 , \ \ \partial u / \partial n = - c
\ on \ \partial D , \label{3.2}
\end{equation}
for some $u \in C^2 (D) \cap C^1 (\overline D)$ and constant $c$. Then $D$ is a
ball and $u = (b^2 - r^2) / (2 m)$, where $b$ and $r$ denote the radius of the ball
and the distance from its centre, respectively.
\end{theorem}


Now we are in a position to present Bennett's proof of Theorem 3.2.


\begin{proof}[Proof of Theorem 3.2.]
It is clear that there exists $u \in C^2 (D) \cap C^1 (\overline D)$ satisfying the
first and second relations \eqref{3.2}. Then the first relation yields
\[ \int_{D} h \, \D x = - \int_{D} h \, \nabla^2 u \, \D x = - \int_{D} u \, \nabla^2 
h \, \D x + \int_{\partial D} \left[ u \frac{\partial h}{\partial n} - h
\frac{\partial u}{\partial n} \right] \D S_x = - \int_{\partial D}  h \frac{\partial
u}{\partial n} \, \D S_x \, ,
\]
where the last equality follows by harmonicity of $h$ and the second relation
\eqref{3.2}. Combining this and \eqref{3.1}, we obtain
\[ \int_{\partial D} h \left[ \frac{\partial u}{\partial n} + \frac{|D|}{|\partial
D|} \right] \D S_x = 0 \, .
\]
Moreover, there exists $h$ harmonic in $D$ and satisfying the condition
\[ h = \frac{\partial u}{\partial n} + \frac{|D|}{|\partial D|} \ \ \mbox{on} \ 
\partial D \, .
\]
Substituting it into the last integral, we obtain that the third condition
\eqref{3.2} is also valid for $u$ with $c = |D| / |\partial D|$. Then $D$ is a ball
by Theorem~3.3.
\end{proof}


Payne and Schaefer \cite{PS} discovered an alternative proof of a slightly modified
version of Theorem~3.2. Namely, instead of \eqref{3.2} it is required that there
exists a constant $c$ such that
\begin{equation}
\int_{D} h (x) \, \D x = c \int_{\partial D} h (x) \, \D S_x \label{3.3}
\end{equation}
holds for every $h$ harmonic in $D$. The proof is based on considerations of
Weinberger's note \cite{W} involving properties of $|\nabla u|^2 + 2 u / m$, where
$u$ solves \eqref{3.2}.

One more application of the integral equality \eqref{3.3} was given by Didenko and
Emami\-zadeh \cite{DE}, whose approach to characterization of balls is as follows.
Let $V : \RR^m \to \RR^m$ be a given vector field with $C^2$-components. Then the
problem
\begin{equation}
\nabla^2 h = 0 \ \ \mbox{in} \ D , \quad h = - (V \cdot n) \, \frac{\partial
h}{\partial n} \ \ \mbox{on} \ \partial D \label{3.4}
\end{equation}
has a solution because $h$ is the so-called domain derivative in the direction of
$V$ of a unique solution to Saint-Venant's problem in $D$; the latter includes the
first two relations \eqref{3.2}. A detailed treatment of the concept of
differentiation with respect to the domain is given in \cite{Si}.

Let $\mathcal V$ denote the set of harmonic functions in $D$, each being a solution
of problem \eqref{3.4} for some possible vector field $V$. Then the result
established in \cite{DE} is as follows.


\begin{theorem} 
Let $D \subset \RR^m$, $m \geq 2$, be a bounded $C^2$-domain. Then $D$ is a ball if
and only if equality \eqref{3.3} is valid for every function $h \in C^2 (D) \cap C^1
(\overline D)$ belonging to $\mathcal V$. Moreover, $- m \, [\partial h / \partial
n]_{\partial D}$ is constant and equals to ball's radius.
\end{theorem}


The question whether $\mathcal V$ is a proper subset of the whole set of functions
harmonic in~$D$, in which case Theorem~3.4 improves the result of Payne and
Schaefer, was not considered in \cite{DE}.

\subsection{Panharmonic functions}

A characterization of balls via panharmonic functions was recently obtained by the
author \cite{Ku,Ku1}. Before giving the precise formulation of the result, we define
a dilated copy of a bounded domain $D$ as follows: $D_r = D \cup \left[ \cup_{x \in
\partial D} B_r (x) \right]$. Thus, the distance from $\partial D_r$ to $D$ is
equal to $r$. The following general assertion was proved in~\cite{Ku1}.

\begin{theorem}
Let $D \subset \RR^m$, $m \geq 2$, be a bounded domain, whose complement is
connected, and let $r$ be a positive number such that $|B_r| \leq |D|$. Suppose that
there exists a point $x_0 \in D$ such that for some $\mu > 0$ the mean value
identity $u (x_0) \, a^\bullet (\mu r) = M^\bullet (u, D)$ holds for every positive
function $u$ satisfying equation \eqref{Hh} in $D_r$. If also $|D| = |B_r|$ provided
$B_r (x_0) \setminus \overline D \neq \emptyset$, then $D = B_r (x_0)$.
\end{theorem}

Unfortunately, the  assumption that the complement of $D$ is connected is missed in
the formulation of this result in \cite{Ku1} and in Theorem~1, \cite{Ku}.

Prior to proving this theorem, let us consider some properties of the function
\begin{equation*}
U (x) = a^\circ (\mu |x|) \, , \quad x \in \RR^m , \label{U}
\end{equation*}
The second formula \eqref{diff} shows that this spherically symmetric function
monotonically increases from unity to infinity as $|x|$ goes from zero to infinity.
A consequence of representation \eqref{PI}, which is easy to differentiate, is that
$U$ solves equation \eqref{Hh} in $\RR^m$. Since both formulae \eqref{cibu} are
similar, Poisson's integral allows us to compare these functions, and the inequality
\begin{equation} 
[ U (x) ]_{|x| = r} > a^\bullet (\mu r) \label{aU} 
\end{equation} 
immediately follows.

\begin{proof}[Proof of Theorem 3.5.] 
Without loss of generality, we suppose that the domain $D$ is located so that $x_0$
coincides with the origin. Let us show that the assumption that $D \neq B_r (0)$
leads to a contradiction.

It is clear that either $B_r (0) \subset D$ or $B_r (0) \setminus \overline D \neq
\emptyset$ (the equality $|B_r| = |D|$ is assumed in the latter case), and we treat
these two cases separately. Let us consider the second case first, for which
purpose we introduce the bounded open sets
\[ G_i = D \setminus \overline{B_r (0)} \quad \mbox{and} \quad G_e = B_r (0)
\setminus \overline D \, ,
\]
whose nonzero volumes are equal in view of the assumptions about $D$ and $r$. The
volume mean identity for $U$ over $D$ can be written as follows:
\begin{equation}
|D| \, a^\bullet (\mu r) = \int_D U (y) \, \D y \, ; \label{1}
\end{equation}
here the condition $U (0) = 1$ is taken into account. Since formula \eqref{Mpb} is
valid for $U$ over $B_r (0)$ ($R=r$ in \eqref{Mpb} in this case), we write it in the
same way:
\begin{equation}
|B_r| \, a^\bullet (\mu r) = \int_{B_r (0)} U (y) \, \D y \, . \label{2}
\end{equation}
Subtracting \eqref{2} from \eqref{1}, we obtain
\begin{equation*}
0 = \int_{G_i} U (y) \, \D y - \int_{G_e} U (y) \, \D y > 0 \, .
\end{equation*}
Indeed, the difference is positive since $U (y)$ (positive and monotonically
increasing with~$|y|$) is greater than $[U (y)]_{|y| = r}$ in $G_i$ and less than
$[U (y)]_{|y| = r}$ in $G_e$, whereas $|G_i| = |G_e|$. This contradiction proves
the result in this case.

In the case when $B_r (0) \subset D$, a contradiction must be deduced when $B_r (0)
\neq D$, that is,  $|G_i| = |D| - |B_r| > 0$. Now, subtracting \eqref{2} from
\eqref{1}, we obtain
\begin{equation*}
( |D| - |B_r| ) \, a^\bullet (\mu r) = \int_{G_i} U (y) \, \D y > |G_i| \, [ U (y)
]_{|y| = r} \, , \label{3}
\end{equation*}
where the last inequality is again a consequence of positivity of $U (y)$ and its
monotonicity. This yields that $a^\bullet (\mu r) > [ U (y) ]_{|y| = r}$, which
contradicts \eqref{aU}. The proof is complete.
\end{proof} 

Let us comment on Theorem 3.5. First, the domain $D$ is supposed to be bounded
because it is easy to construct an unbounded domain of finite volume in which $U$ is
not integrable. Thus, the boundedness of $D$ allows us to avoid formulating rather
complicated restrictions on the domain.

Second, one obtains Laplace's equation from \eqref{Hh} in the limit $\mu \to 0$, and
the assumption about $r$ becomes superfluous in this case. Therefore, Theorem 3.5
turns into an improved version of Kuran's because only positive harmonic functions
are involved; see also \cite{AG}.

Furthermore, the integral $\int_D u (y) \, \D y$ can be replaced by the flux
$\int_{\partial D} \partial u / \partial n_y \, \D S_y$ in the formulation of
Theorem~3.5 provided $\partial D$ is sufficiently smooth; here $n$ is the exterior
unit normal. Indeed, we have
\[ \int_D u (y) \, \D y = \mu^{-2} \int_D \nabla^2 u \, (y) \, \D y = \mu^{-2}
\int_{\partial D} \partial u / \partial n_y \, \D S_y \, .
\]
These relations are used in \cite{Ku1}; see the proof of Theorem 9, which
characterizes solutions of equation \eqref{Hh} in terms of the mean flux through
spheres. This suggests the following.

\begin{conjecture}
Let $D \subset \RR^m$, $m \geq 2$, be a bounded domain with sufficiently smooth
boundary, and let $r > 0$ be such that $|B_r| = |D|$. If there exists $x_0 \in D$
such that for some $\mu > 0$ the equality
\begin{equation*}
\frac{2 \mu}{m} \, a^\bullet (\mu r) \, u (x_0) = \frac{1}{|\partial D|}
\int_{\partial D} \frac{\partial u}{\partial n_y} \, \D S_y \label{c_1}
\end{equation*}
holds for every panharmonic $u \in C^1 (\overline D)$, then $D = B_r (x_0)$.
\end{conjecture}

\vspace{-4mm}

{\centering \section{Characterization of annuli and strips via mean values} }

\vspace{-2mm}

\subsection{Annuli}

It was Avici \cite{Av}, who attempted to prove an assertion inverse to Theorem~2.3
in 1981. However, the statement and its proof are both erroneous in \cite{Av}; see
the comments by Armitage and Goldstein \cite{AG}, pp.~142, 145. The correct
formulation (see the recent paper \cite{GS} by Gardiner and Sj\"odin), involving the
radius~$r_*$ given by \eqref{r*} for $m \geq 3$ and by \eqref{r*2} for $m = 2$, is
as follows.

\begin{theorem}
Let $D \subset \RR^m$, $m \geq 2$, be an open set such that $|D| < \infty$ and
$\partial B_{r_*} \subset D$. If any function $u$ harmonic in $D$ and integrable
over $D$ satisfies the identity
\begin{equation}
M^\bullet (u, D) = M^\circ (u, \partial B_{r_*}) \, , \label{MDa}
\end{equation}
then either $D$ is the $r_*$-annulus with $0 \leq r_1 < r_2$ or $D$ is an open ball
centred at the origin.
\end{theorem}

Already in 1989, Armitage and Goldstein \cite{AG} obtained a weaker result, namely:
identity \eqref{MDa} implies that $\bar D$ is either the closed $r_*$-annulus or the
closed ball centred at the origin. Later they asked, in Problem 3.35 of \cite{HL},
whether a similar assertion is true when only a fundamental solution of Laplace's
equation is involved in \eqref{MDa}. An improved version of that problem involving
the following solution
\[ g (x,y) = \phi_m (|x-y|) , \ \ \mbox{where} \ \phi_m (t) = t^{2-m} \ \mbox{for} \ m 
\geq 3 \ \mbox{and} \ \phi_2 (t) = - \log t \, ,
\] 
was formulated by Gardiner and Sj\"odin \cite{GS} as follows.

\begin{theorem}
Let $D \subset \RR^m$, $m \geq 3$, be an open set such that $|D| < \infty$ $(let \
D$ be bounded in two dimensions$)$ and $\partial B_{r_*} \subset D$. If $g (x,y)$
satisfies identity \eqref{MDa} for any $y \in \RR^m \setminus D$, then either $D$ is
the $r_*$-annulus with $0 \leq r_1 < r_2$ or $D = B \setminus T$, where $B$ is a
ball centred at the origin and $T \subset \partial B_{r_0}$ for some $r_0 < r_*$
$(T$ may be empty$)$.
\end{theorem}

Their proof of this theorem is based on the main result of Hansen and Netuka
\cite{HN} dealing with a mean value property of $U_E (y) = \int_E g (x,y) \, \D x$,
where $E \subset \RR^m$ is Lebesgue measurable and $|E| \in (0, \infty)$.

\begin{theorem}
Let $B \subset \RR^m$, $m \geq 2$, be the open ball centred at the origin and such
that $|B| = |E| \in (0, \infty)$. If for any compact $F, G \subset \RR^m \setminus
E$ the identity
\begin{equation}
|E|^{-1} \int_E \left[ U_F (y) - U_G (y) \right] \D y = U_F (0) - U_G (0)
\label{4.3}
\end{equation}
is valid provided $U_F - U_G$ is bounded, then $|B \setminus E| = 0$.
\end{theorem}

In particular, the proof of this theorem implies that the requirement $U_E (y) = |E|
\, g (0,y)$ almost everywhere outside $E$ may be used instead of identity
\eqref{4.3} in the last theorem provided $m \geq 3$; for $m = 2$ the analogous
requirement is slightly more complicated.

Gardiner and Sj\"odin's derivation of Theorem 4.2 is highly technical and rather
long. Subsequently, in order to prove Theorem 4.1 it remains to show that the set
$T$ permissible by Theorem 4.2 is empty. This is achieved by choosing some
particular harmonic function and demonstrating that it satisfies an inequality which
contradicts \eqref{MDa}.

An alternative approach to characterizations of two-dimensional annuli was developed
by Rodr\'igues \cite{R}; in some sense, it is similar to Bennett's \cite{Be} for
balls (see Theorems~3.2 and 3.3 above). She assumes that $D \subset \RR^2$ is a
bounded, finitely connected domain, whose boundary $\partial D$ consists of two or
more closed analytic curves which are pairwise disjoint. Let $\Gamma_0$ denote the
curve separating $D$ from infinity and let $\Gamma_1 = \partial D \setminus
\Gamma_0$; the simply connected domain within $\Gamma_0$ is denoted by $D_0$ and
$D_1 = D_0 \setminus \bar D$; also, let $c_0 = |D_0| / |\Gamma_0|$ and $c_1 = -
|D_1| / |\Gamma_1|$. Now, we are in a position to formulate the result obtained in
\cite{R}.

\begin{theorem}
Let $D \subset \RR^2$ be a bounded domain and $\partial D$ consists of a finite
number (two or more) closed analytic curves which are pairwise disjoint. Then the
following asser\-tions are equivalent:

\noindent {\rm (i)} For some $u \in C^2 (D) \cap C^1 (\overline D)$ the relations
\begin{equation*}
\nabla^2 u = - 1 \ in \ D ; \ \ u = - c_0^2 , \ \partial u / \partial n = - c_0 \ on
\ \Gamma_0 ; \ \ u = - c_1^2 , \ \partial u / \partial n = - c_1 \ on \ \Gamma_1 
\end{equation*}
are fulfilled with $c_0$ and $c_1$ defined above.

\noindent {\rm (ii)} The quadrature identity
\[ \int_D h \, \D x = c_0 \int_{\Gamma_0} h \, \D S + c_1 \int_{\Gamma_1} h \, \D S
- c_0^2 \int_{\Gamma_0} \partial h / \partial n \, \D S - c_1^2 \int_{\Gamma_1}
\partial h / \partial n \, \D S
\]
is valid for every harmonic in $D$ function $h \in C^1 (\overline D)$.

\noindent {\rm (iii)} $D$ is an annulus centred at the origin, whose smaller radius
is $- 2 c_1$ and the larger one is $2 c_0$.
\end{theorem}

As in the case of the quadrature identity \eqref{3.3} yielding that $D$ is a ball,
the proof is based on considerations of Weinberger's note \cite{W} involving
properties of $|\nabla u|^2 + 2 u / m$, where $u$ satisfies the first relation
\eqref{3.2}.

\subsection{Strips}

Some early partial results concerning characterization of strips by harmonic
functions were surveyed in 1992 in a conference proceedings; see \cite{AG3} and
\cite{GHR2}. The following general inverse of Theorem~2.4 was obtained by Armitage
and Nelson \cite{AN} next year. 

\begin{theorem}
Let $D \subset \RR^{m+1}$, $m \geq 2$, be an open subset of $S (-a, a)$ for some $a
\in (0, \infty)$ and let $\{0\} \times \RR^m$ be a proper subset of $D$. If the
identity\\[-3mm]
\begin{equation}
\int_D h (t, x) \, \D t \D x = 2 \int_{\RR^m} h (0, x) \, \D x \label{AN}
\end{equation}
is valid for every positive $h$ harmonic in $D$ and integrable over $D$, then $D = S
(-1, 1)$.
\end{theorem}

The requirement that $D \subset S (-a, a)$ is essential because \eqref{AN} is
fulfilled vacuously when $D = S (c, +\infty)$ with $c \in [-\infty, 0)$.

This theorem is an improvment of that obtained by Goldstein, Haussmann and Rogge
\cite{GHR1}, who used the following additional assumptions:\\[-5mm]
\begin{eqnarray*} 
&& \ \ \ \ \ \ \ \ \ \ \ \ \ \ \ \ \ \ \ D \subset S (-3, 3) \, ; \ \ \ \ \partial S
(-1, 1) \setminus D \neq \emptyset \, ; \\ && \mathrm{\eqref{AN} \ is \ fulfilled \
for \ all \ harmonic \ in} \ D \ \mathrm{functions \ integrable \ over} \ D .
\end{eqnarray*}
Without these assumptions, Armitage and Nelson's proof of Theorem~4.5 is completely
different from that in \cite{GHR1}, being based on the refined technique applied
earlier in \cite{AG4}. 

Let us outline their proof for $m \geq 2$. It involves an investigation of the Green
potential on the half-space $S (b, +\infty)$, where $b = -a - 1$. The Green kernel
of $S (b, +\infty)$ is\\[-3mm]
\begin{equation}
G (t,x; \tau, \xi) = |(t,x) - (\tau, \xi)|^{1-m} - |(t,x) - (\tau, \xi)^*|^{1-m} , \
\ \ (t,x) , (\tau, \xi) \in S (b, +\infty) \, , \label{2.1}
\end{equation}
where $|(t,x) - (\tau, \xi)|^2 = |t - \tau|^2 + |x - \xi|^2$ and $(\tau, \xi)^*$
denotes the mirror-image of $(\tau, \xi)$ with respect to the hyperplane $\{b\}
\times \RR^m$; that is, $(\tau, \xi)^* = (2 b - \tau, \xi_1, \dots, \xi_m)$. The
Green potential is defined by\\[-3mm]
\[ U (t, x) = \int_D G (t,x; \tau, \xi) \, \D \tau \D \xi \, , \ \ (t, x) \in S (b,
+\infty) . 
\]
The properties of $U$ used in the proof are as follows.

\begin{lemma}[Doob \cite{Do}, 1.I.7]
\rm{(i)} $U \in C^2 (D) \cap C^1 ( S (b, +\infty) );$ \rm{(ii)} $|\nabla_x U|$ is
bounded on $D;$ \rm{(iii)} $U_{tt} + \nabla_x^2 U = (m^2 - 1) \, \omega_{m+1}$ in
$D$.
\end{lemma}

Armitage and Nelson begin their proof by demonstrating that $D$ is connected. For
this purpose serves the function $G (t,x; a, 0) (\chi_{D_c} + k \chi_{D \setminus
D_c})$ with $k=1$ and $k=2$; here $\chi_E$ is the characteristic function of a set
$E$ and $D_c$ is a connected component of $D$. It is clear that these two functions
cannot satisfy \eqref{AN} simultaneously. Extensively using Lemma~4.1, it is shown
on the next step that there exists a real number $g$ such that $D \subset S (g - 2,
g)$ and $|S (g-2, g) \setminus D| = 0$. Finally, the expression for\\[-3mm]
\[ \int_{\RR^m} G (t,x; c, \xi) \, \D \xi \, , \ \ \mbox{where} \ c \in (b, +\infty),
\]
(it was obtained in \cite{AG2}) is applied for demonstrating that $g = 1$ and that
the inclusion is in fact an equality. The same proof is valid for $m=1$ provided
$|(t,x) - (\tau, \xi)|^{1-m}$ is changed to $-\log |(t,x) - (\tau, \xi)|$.

Now we turn to a characterization of strips that is similar to Theorem~3.1 for
balls, but involves the Green kernel \eqref{2.1} instead of the fundamental
solution of Laplace's equation.

\begin{theorem}
Let $D$ be an open set such that $\{0\} \times \RR^m \subset D \subset S (-a, a)$
for some $a \in (0, \infty)$. If for $G$ defined by \eqref{2.1} the identity
\begin{equation*}
\int_D G (t,x; \tau, \xi) \, \D \tau \D \xi = 2 \int_{\RR^m} G (t,x; 0, \xi) \, \D
\xi \, ,
\end{equation*}
is valid for every $(t,x) \in [(-a-1, \infty) \times \RR^m] \setminus D$, then $D = S (-1,
1)$.
\end{theorem}

The proof of this assertion in \cite{NAW}, pp.~247--248, repeats to a large extent
Armitage and Nelson's proof of Theorem~4.5; see above.

It occurs that a bi-infinite cylinder is also characterized by harmonic quadrature;
namely, the following inverse of Theorem~2.5 was obtained by Goldstein, Haussmann,
Rogge \cite{GHR}.

\begin{theorem}
Let $D \subset \RR^{m+l}_{x,y}$ be an open subset of some (arbitrarily large)
cylinder such that $\{ (0,\dots,0) \} \times \RR^l \subset D$ and $\mathrm{int} \bar
D = D$. If for every positive $h$ is harmonic and integrable on $D$ we have
\begin{equation*}
\int_{D} h (x, y) \, \D x \D y = |B_r| \int_{\RR^l} h (0, y) \, \D y \, ,
\end{equation*}
then $D = B_r \times \RR^l$.
\end{theorem}

The proof of this theorem in \cite{GHR} is based on a relationship between Green's
functions of $B_r \times \RR^l$ and of $B_r$ (its tedious derivation occupies four
pages) and on the following result of independent interest.

\begin{proposition}
Let $D \subset \RR^m$, $m \geq 2$, be an unbounded open set such that 
\begin{equation}
|D \cap B_r| = o (r^m) \ \ \ as \ r \to \infty . \label{fin}
\end{equation}
If $h \in C (\bar D) \cap L^\infty (D)$ is harmonic in $D$ and vanishes on
$\partial D$, then $h$ vanishes on~$\bar D$.
\end{proposition}

To prove this assertion the authors estimate the mean value of the subharmonic
function $|h|$ using property \eqref{fin}.

\vspace{-10mm}

\renewcommand{\refname}{
\begin{center}{\Large\bf References}
\end{center}}
\makeatletter
\renewcommand{\@biblabel}[1]{#1.\hfill}
\makeatother

\end{document}